\documentclass[11pt]{article}

\usepackage[dvips,a4paper,hmargin=2.75cm,vmargin=3cm]{geometry}
\usepackage[utf8]{inputenc}
\usepackage{hyperref}
\usepackage{amssymb,amsmath,amsthm,mathtools}
\usepackage{enumitem}
\usepackage{cite}
\usepackage[mathscr]{euscript}
\usepackage[T1]{fontenc}
\usepackage{lmodern}
\usepackage{fontawesome}
\usepackage{microtype}
\usepackage{graphicx}
\usepackage[normalem]{ulem}
\usepackage{authblk}
\usepackage[small]{titlesec}

\newcommand{\mailto}[1]{\href{mailto:#1}{\nolinkurl{#1}}}

\hypersetup{colorlinks=true,linkcolor=blue,citecolor=blue,
	        pdfinfo={
	        Title   = {Oversampling and aliasing in de Branges 
					   spaces arising from Bessel operators},
	        Author  = {Julio H. Toloza, Alfredo Uribe},
	        Subject = {Manuscript}}
	       }

\titleformat*{\subsection}{\bfseries\itshape}

\providecommand{\keywords}[1]
{
\noindent{\small
	\textbf{Keywords:} #1}
}

\providecommand{\msc}[1]
{
\noindent{\small
	\textbf{2010 MSC:} #1}
}

\newtheorem{theorem}{Theorem}[section]
\newtheorem{lemma}[theorem]{Lemma}
\newtheorem{proposition}[theorem]{Proposition}
\newtheorem{corollary}[theorem]{Corollary}
\newtheorem*{thm}{Theorem}

\theoremstyle{definition}

\newtheorem{remark}[theorem]{Remark}

\newcommand{\R}{{\mathbb R}}
\newcommand{\N}{{\mathbb N}}
\newcommand{\Z}{{\mathbb Z}}
\newcommand{\C}{{\mathbb C}}
\newcommand{\K}{{\mathbb K}}

\newcommand{\cC}{\mathcal{C}}
\newcommand{\cB}{\mathcal{B}}
\newcommand{\cD}{\mathcal{D}}

\newcommand{\cO}{\mathcal{O}}
\newcommand{\cPW}{\mathcal{PW}}
\newcommand{\cR}{R}

\DeclareMathOperator{\im}{Im}
\DeclareMathOperator{\dom}{\cD}

\DeclareMathOperator{\assoc}{assoc}

\newcommand{\abs}[1]{\left\lvert #1 \right\rvert}
\newcommand{\abss}[1]{\lvert #1 \rvert}
\newcommand{\norm}[1]{\left\lVert #1 \right\rVert}
\newcommand{\inner}[2]{\left\langle#1,#2\right\rangle}

\newcommand{\cc}[1]{\overline{#1}}

\renewcommand\tilde{\widetilde}

\newcommand{\defeq}{\mathrel{\mathop:}=}

\title {\bf Oversampling and aliasing in de Branges spaces arising from Bessel 
			operators}
\author[1]{Julio H. Toloza%
			\thanks{\faEnvelopeO\, {julio.toloza@uns.edu.ar}}%
			}
\author[2]{Alfredo Uribe%
			\thanks{\faEnvelopeO\, {alfredo.uribe.83@gmail.com}}%
			}
\affil[1]{INMABB\\ 
		 Departamento de Matemática\\
		 Universidad Nacional del Sur (UNS) - CONICET\\
		 Bahía Blanca\\
		 Argentina}
\affil[2]{Departamento de Matemáticas\\
		 Universidad Autónoma Metropolitana\\
		 Av. San Rafael Atlixco 186\\
		 Col. Vicentina, Iztapalapa, C.P. 09340, México D.F.}

\begin{document}

\date{}
\maketitle

%\begin{abstract}
%We show that de Branges spaces generated by symmetric singular differential
%operators of the Bessel type have the properties of oversampling and aliasing.
%\end{abstract}

\begin{abstract}
We show that a class of de Branges spaces, generated by means of generalized Fourier 
transforms associated with perturbed Bessel differential equations, 
has the properties of oversampling and aliasing.
\end{abstract}

\bigskip

\msc{
46E22, %  	Hilbert spaces with reproducing kernels (= [proper] functional Hilbert 
%33E30, %	Other functions coming from differential, difference and integral equations
42C15, %   	General harmonic expansions, frames
34L40, %   	Particular operators (Dirac, one-dimensional Schrödinger, etc.)
94A20  %   	Sampling theory
%47B32   	Operators in reproducing-kernel Hilbert spaces (including de Branges, 
%           de Branges-Rovnyak, and other structured spaces)
%33C10   	Bessel and Airy functions, cylinder functions
%34L20   	Asymptotic distribution of eigenvalues, asymptotic theory of eigenfunctions
}
%%%%%%%%%%%%%%%%%%%%%%%%%%%%%%%%%%%%%%%%%%%%%%%%%%%%%%%%%%%%%%%%%%%%%%%%%%%%%%%%

\section{Introduction and main results}
\label{sec:intro}

\subsection{A hint at de Branges spaces}

An entire function $E$ belongs to the Hermite-Biehler class if it has the property 
$\abs{E(z)}>\abs{E(\cc{z})}$ for all $z\in\C_+$. Given such a function,
let us define
\begin{equation}
\label{eq:reproducing-kernel}
K(z,w)\defeq
		\begin{cases}
		\frac{E^\#(z)E(\cc{w})-E(z)E^\#(\cc{w})}{2\pi i(z-\cc{w})},
					& z\ne\cc{w},
		\\[1mm]
		\frac{{E^\#}'(z)E(z)-E'(z)E^\#(z)}{2\pi i},
					& z=\cc{w}.			
		\end{cases}
\end{equation}
The de Branges spaces generated by $E$ is the linear set
\begin{equation*}
\cB(E) \defeq \left\{F\text{ entire}:\norm{F}^2\defeq\int_{-\infty}^{\infty}
				\abs{\frac{F(\lambda)}{E(\lambda)}}^2d\lambda<\infty,\;
				\abs{F(z)}^2\le \norm{F}^2 K(z,z)\text{ for all } z\in\C
		 \right\}
\end{equation*}
equipped with the inner product
\begin{equation*}
\inner{F}{G}_\cB \defeq 
	\left(\int_{-\infty}^{\infty}
					\frac{\cc{F(\lambda)}G(\lambda)}
					{\abs{E(\lambda)}^2}d\lambda\right)^{1/2}.
\end{equation*}
$\cB(E)$ is a reproducing kernel Hilbert space whose reproducing kernel is 
precisely \eqref{eq:reproducing-kernel} \cite[Thm.~20]{debranges}.
From now, on we will denote a de Branges space and its associated reproducing
kernel as $\cB$ and $K_\cB(z,w)$, respectively. 
In passing, we note that there are alternative ways of defining a de 
Branges space \cite{debranges,remling,zbMATH06526214}.
Also, there are many Hermite-Biehler functions that generate a given de Branges space
\cite[Thm. 1]{debranges0}.

Let $S_\cB:\dom(S_\cB)\to\cB$ denote the operator defined by
\[
\dom(S_\cB) = \{F\in\cB:zF(z)\in\cB\},
\qquad (S_\cB F)(z) \defeq zF(z).
\]
It is well known that $S_\cB$ is a regular, closed, symmetric operator with 
deficiency indices $(1,1)$ \cite[Prop.~4.2 and Lemma~4.7]{kaltenback}.

Let $S_{\cB,\gamma}$, $\gamma\in[0,\pi)$, denote the canonical 
selfadjoint 
extensions of $S_\cB$ (viz., selfadjoint restrictions of $S_\cB^*$).
Due to the regularity of $S_\cB$, the spectra $\sigma(S_{\cB,\gamma})$ consist 
of isolated eigenvalues of multiplicity equal to one, that moreover satisfy
\[
\bigcup_{\gamma\in[0,\pi)}\sigma(S_{\cB,\gamma})=\R,\quad
\sigma(S_{\cB,\gamma})\cap\sigma(S_{\cB,\gamma'})=\emptyset,\quad\gamma\ne\gamma'.
\]

It is straightforward to verify that $K_\cB(z,w)\in\ker(S^*_\cB-\cc{w}I)$ for 
all $w\in\C$. It follows that 
$\{K_\cB(z,\lambda)\}_{\lambda\in\sigma(S_{\cB,\gamma})}$ is an 
orthogonal basis. Hence, the sampling formula
\begin{equation}
\label{eq:sampling-general}
F(z) = \sum_{\lambda\in\sigma(S_{\cB,\gamma})}
	\frac{K_\cB(z,\lambda)}{K_\cB(\lambda,\lambda)}F(\lambda)
\end{equation}
holds true for all $F\in\cB$. The convergence of this series is in the 
norm, which in turn implies uniform convergence in compact subsets of 
$\C$. 
The Parseval-Plancherel identity implies that
any sequence $\{\delta_\lambda\}_{\lambda\in\sigma(S_{\cB,\gamma})}$ obeying
\begin{equation*}
\sum_{\lambda\in\sigma(S_{\cB,\gamma})}
	\frac{\abs{\delta_\lambda}^2}{K_\cB(\lambda,\lambda)} < \infty
\end{equation*}
yields an approximation to $F$ by means of the formula 
\eqref{eq:sampling-general}, when the samples
$\{F(\lambda)\}_{\lambda\in\sigma(S_{\cB,\gamma})}$ are replaced by
$\{F(\lambda) + \delta_\lambda\}_{\lambda\in\sigma(S_{\cB,\gamma})}$.
In other words, \eqref{eq:sampling-general} is stable under weighted 
$\ell_2$-perturbations.

There is a distinctive structural property of de Branges spaces
related to the subject of this paper.
For the class of spaces discussed here, this
property can be stated as follows: Assume $\cB_1$, $\cB_2$ and $\cB_3$ 
are de Branges spaces such that $\cB_1$ and $\cB_2$ are both isometrically 
contained in $\cB_3$. Then either $\cB_1\subset\cB_2$ or $\cB_2\subset\cB_1$.
A more general form of this assertion is in the classical book by
de Branges \cite[Thm.~35]{debranges}.

\subsection{Main results}
%\newstuff

The class of de Branges spaces considered in this work are defined by means of
generalized Fourier transforms associated with perturbed Bessel 
differential equations. Namely,
\begin{equation*}
%\label{eq:associated-dB-space}
\cB_{s} 
	:= \left\{F(z)=\int_0^s \xi(z,x)\varphi(x)dx : \varphi \in 
	L^2(0,s)\right\},\quad
\norm{F}^2_{\cB_{s}} 
		= \int_0^s \abs{\varphi(x)}^2dx,
\end{equation*}
where $\xi(z,x)$ is the real entire solution (with respect to $z$) to the 
eigenvalue problem 
\begin{equation*}
-\varphi'' + \left(\frac{\nu^2-1/4}{x^2} + q\right)\varphi = z\varphi,
\quad x\in(0,\infty),\quad \nu\in(0,\infty),\quad z\in\C,
\end{equation*}
subject to the boundary condition
\begin{equation*}  
%\label{eq:boundary-left-condition-limit-circle}
\lim_{x \to 0+}x^{\nu-\frac12}\left((\nu+1/2)\varphi(x)-x\varphi'(x)\right)=0
\end{equation*}
when $\nu\in(0,1)$. A Hermite-Biehler function that generates $\cB_s$ is
\begin{equation*}
E_s(z) = \xi(z,s) + i\xi'(z,s), 
\end{equation*}
where $'$ denotes derivative with respect to $x$.
In this framework, our main results can be summarized
as follows:

\begin{thm}[oversampling]
Fix $\nu,b\in(0,\infty)$, $a\in(0,b)$ and $\gamma\in[0,\pi)$. 
Assume that $q$ is a real-valued 
function belonging to $\text{AC}_\text{loc}(0,b]$ such that 
$xq(x)\in L^r(0,b)$ for some $r\in(2,\infty]$.
Given $\epsilon=\{\epsilon_n\}\in\ell_\infty(\nu)$ and $F\in\cB_a$,
define
\begin{equation*}
F_\epsilon(z)=\sum_{\lambda_n\in\sigma(S_{b,\gamma})}
		\frac{J_{ab}(z,\lambda_n)}
		{K_b(\lambda_n,\lambda_n)}\left(F(\lambda_n)+\epsilon_n\right),
\end{equation*}
where $\ell_\infty(\nu)$ and $J_{ab}(z,w)$ are defined by 
\eqref{eq:ell-simple} and \eqref{eq:function-Jab} respectively.
Then, for every compact set $\K\subset\C$, there exists $C>0$ 
such that
\[
\abs{F(z)-F_\epsilon(z)} \le C\norm{\epsilon}_{\ell_\infty(\nu)},
	\quad z\in \K,
\]
uniformly for all $F\in\cB_a$.
\end{thm}

\begin{thm}[aliasing]
Suppose the same hypotheses of the previous theorem, except that
$\gamma\in(0,\pi)$. For every $F\in\cB_b$, define
\begin{equation*}
\tilde{F}(z)
= \sum_{\lambda_n\in\sigma(S_{a,\gamma})}
	\frac{K_a(z,\lambda_n)}{K_a(\lambda_n,\lambda_n)}\, F(\lambda_n).
\end{equation*}
Then, for each compact set $\K\subset\C$, there exists $D>0$ such that
\begin{equation*}
\abs{F(z)-\tilde{F}(z)}
	\le D\norm{(I-P_{ab})F}_{\cB_b},\quad z\in\K,
\end{equation*}
where $P_{ab}:\cB_b\to\cB_b$ is the orthogonal projector onto $\cB_a$. 
\end{thm}

It is known that $\cB_s=\cB_{\nu,s}$ setwise, where 
$\cB_{\nu,s}$ is the de Branges space associated with $q\equiv 0$ (and 
the same value of $\nu$), also under the hypothesis $xq(x)\in L^r(0,s)$ with 
$r\in(2,\infty]$
\cite[Thm.~4.2]{siltol2}.
Since it is natural to consider sampling formulas as regular for the case 
$q\equiv 0$ and irregular otherwise, our theorems above 
provide estimates for oversampling and aliasing error on a (relatively 
restricted) set of irregular sampling formulas for regularized Hankel transform 
of functions with compact support
in $\R_+$. The restriction $r>2$ is a technical limitation due to the 
perturbation methods
involved in both \cite{siltol2} and the present paper; we believe that the 
results exposed here should hold under the weaker assumption
$xq(x)\in L^1(0,s)$.

\subsection{A bit of history}
%\newstuff

The notions of oversampling and aliasing
stem from the theory of Paley-Wiener spaces 
\cite{MR1473224,MR1270907}, that is, the spaces of Fourier 
transform of functions with given compact support centered at zero,
\begin{equation*}
	\cPW_a \defeq \left\{F(z)=\int_{-a}^a e^{-ixz}\phi(x)dx : \phi\in 
	L^2(-a,a)\right\}.
\end{equation*}
The linear set $\cPW_a$, equipped with the norm $\norm{F}=\norm{\phi}$, is
a de Branges space generated by the Hermite-Biehler 
function $E_a(z)=\exp(-iza)$.
In this setting, the sampling formula \eqref{eq:sampling-general} is known as the 
Whittaker-Shannon-Kotelnikov theorem, and takes the form
\begin{equation}
\label{eq:wsk-sampling}
	F(z) =\sum_{n\in\Z}
		\mathcal{G}_a\left(z,\frac{n\pi}{a}\right)F\left(\frac{n\pi}{a}\right),
	\qquad
	\mathcal{G}_a\left(z,w\right)
		\defeq\frac{\sin\left[a(z-\cc{w})\right]}{a(z-\cc{w})},
\end{equation}
The function $\mathcal{G}_a\left(z,w\right)$ is referred to as the sampling 
kernel, while the separation between sampling points $\pi/a$ is known as the
Nyquist rate.

As shown in
\cite[Thm. 7.2.5]{MR1473224}, every $F\in\cPW_a\subset\cPW_b$ ($a<b$) admits
the representation
\begin{equation}
  \label{eq:wsk-oversampling}
  	F(z)=\sum_{n\in\Z}\mathcal{G}_{ab}\left(z,\frac{n\pi}{b}\right)
	F\left(\frac{n\pi}{b}\right),
\end{equation}
with the modified sampling kernel 
\begin{equation}
\label{eq:oversampling-PW}
\mathcal{G}_{ab}(z,w)
 \defeq \frac{2}{b-a}\frac{\cos((z-\cc{w})a) - \cos((z-\cc{w})b)}{(z-\cc{w})^2}.
\end{equation}
While the convergence
of the sampling formula \eqref{eq:wsk-sampling} is unaffected by 
$\ell_2$-perturbations of the samples $F\left(\frac{n\pi}{a}\right)$, 
the oversampling formula \eqref{eq:wsk-oversampling} is more stable in 
the sense that it  
converges under $\ell_\infty$-perturbations on the samples.  That is, if the
sequence $\{\delta_n\}_{n\in\Z}$ is bounded and one defines
\begin{equation*}
% \label{eq:wsk-oversampling-with-error}
	F_\delta(z)\defeq
	\sum_{n\in\Z}
	\mathcal{G}_{ab}\left(z,\frac{n\pi}{b}\right)
	\left[F\left(\frac{n\pi}{b}\right)+\delta_n\right],
\end{equation*}
then $\abss{F_\delta(z)-F(z)}$ is uniformly bounded in compact subsets of $\C$ 
and, moreover, uniformly bounded on the real line.
In other words, a more stable sampling formula is obtained at the expense of 
collecting samples at a higher Nyquist rate. 

Aliasing, on the other hand, approximates a function 
$F\in\cPW_b\setminus\cPW_a$ by another one formally constructed using 
the sampling
formula \eqref{eq:wsk-sampling}, namely,
\begin{equation}
\label{eq:wsk-undersampling}
  \widetilde{F}(z)
  	=\sum_{n\in\Z}
   \mathcal{G}_a\left(z,\frac{n\pi}{a}\right)F\left(\frac{n\pi}{a}\right).
\end{equation}
As shown in \cite[Thm.~7.2.9]{MR1473224}, the series in 
\eqref{eq:wsk-undersampling} is indeed convergent and,
moreover,
$\abss{\widetilde{F}(z) - F(z)}$ is uniformly bounded in compact subsets
of $\C$. Formula \eqref{eq:wsk-undersampling} yields in fact an
approximation not only for functions in $\cPW_b\setminus\cPW_a$, but
for the Fourier transform of elements in $L^1(\R)\cap L^2(\R)$.

Generalizations of sampling formula \eqref{eq:wsk-sampling} have been a subject
of research for quite some time within the theory of reproducing 
kernel Hilbert spaces \cite{butzer,higgins1,jorgensen-2016,MR3011977}; 
a classical result is the Kramer's sampling theorem.
The particular case of reproducing kernel Hilbert spaces related
to the Bessel-Hankel transform ($q\equiv 0$ in the context of 
this paper) has been studied by Higgins in \cite{higgins0}.
Sampling theorems associated with Sturm-Liouville problems (hence somewhat related
to the methods involved in this paper) have been discussed in 
\cite{zayed1,zayed2,zayed3}. See also \cite{zayed4}.

Analysis of error due to noisy samples and aliasing in Paley-Wiener spaces goes 
back at least to \cite{papoulis-1966}.  More recent literature on the subject 
is, for instance,
\cite{bailey-2015,bailey-2016,chen,delprete,bodmann-2012,gensun,higgins1,kozachenko-2016,thakur}.
Oversampling in shift-invariant spaces is considered in \cite{garcia0}.

To the best of our knowledge, oversampling and
aliasing on de Branges spaces besides the Paley-Wiener
class have not been discussed until recently, where this subject has been 
touched upon for de Branges spaces associated with regular Schrödinger 
operators using perturbative methods \cite{siltoluri}. In this paper we extend
the results of \cite{siltoluri} to the larger class of de Branges spaces
characterized by \cite[Thm.~4.2]{siltol2} (see Theorem~\ref{thm:invariance-spaces}
below).

\subsection{Organization of this paper}

In Section~\ref{sec:prel} we summarize the necessary results concerning
the perturbation theory of Bessel operators; the main source
is (part of) the work by Kostenko, Sakhnovich and Teschl on a scalar singular
Weyl-Titchmarsh theory \cite{kostenko,kt}. 
Section~\ref{sec:oversampling} is devoted to oversampling, while the 
results concerning aliasing are the subject of Section~\ref{sec:aliasing}.
Some necessary albeit tedious computations are presented in the Appendix.

\section{de Branges spaces arising from Bessel operators}
\label{sec:prel}

\subsection{The unperturbed problem}

Let us consider the differential expression
\begin{equation*}
\tau_\nu \defeq 
	-\frac{d^2}{d x^2}+
			\frac{\nu^2-1/4}{x^2}, \quad x\in(0,s], 
			\quad \nu \in [0,\infty),
\end{equation*}
with $s\in(0,\infty)$. 
It is well-known that $\tau_\nu$ is regular at $x=s$, whereas at $x=0$ 
it is in the limit point case if $\nu\geq 1$ or in the limit circle case 
if $\nu\in [0,1)$. 
The eigenvalue problem $\tau_\nu \varphi = z \varphi$  ($z\in\C$) has linearly 
independent solutions
\begin{align}
\xi_\nu(z,x)
	&= z^{-\frac{\nu}{2}}\sqrt{\frac{\pi x}{2}}J_{\nu}(\sqrt{z}x),
				\label{eq:fundamental-solution-unperturbed-good}
\\[1mm]
\theta_\nu(z,x)
	&= z^{\frac{\nu}{2}}\sqrt{\frac{\pi x}{2}}
		\begin{cases}
		\frac{1}{\sin(\nu\pi)}J_{-\nu}(\sqrt{z}x),& \nu\in\R_+\setminus\N_0,
		\\[1mm]
		\frac{1}{\pi}\log(z)J_{\nu}(\sqrt{z}x) - Y_{\nu}(\sqrt{z}x),
												  & \nu\in\N_0,
		\end{cases}
		\label{eq:fundamental-solution-unperturbed-bad}
\end{align}
where $J_\nu$ and $Y_\nu$ are the Bessel and Neumann functions; $\sqrt{\cdot}$ denotes 
the main branch of the square root. Both solutions
are real entire functions with respect to $z$, with Wronskian
\begin{equation*}
W_x\left(\theta_\nu(z),\xi_\nu(z)\right) 
	\defeq \theta_\nu(z,x)\xi_\nu'(z,x) - \theta_\nu'(z,x)\xi_\nu(z,x) \equiv 1.
\end{equation*}
Moreover, $\xi_\nu(z,\cdot)$ is square-integrable and satisfies the boundary 
condition
\begin{equation}  
\label{eq:boundary-left-condition-limit-circle}
\lim_{x \to 0+}x^{\nu-\frac12}\left((\nu+1/2)\varphi(x)-x\varphi'(x)\right)=0
\end{equation}
when $\nu\in[0,1)$.

\begin{remark}
\label{rem:nu-not-zero}
In order to simplify the discussion, in what follows we shall assume
$\nu>0$. This is because the case $\nu=0$ entails the occurrence of 
logarithmic expressions that would require a somewhat clumsier, 
separated analysis. 
In our opinion, this extra workload would 
not add anything substantial to our results.
\end{remark}

As shown in \cite[Lemmas~A.1 and A.2]{kostenko}, 
\begin{gather}
\label{eq:A1-kost-sak-tes-unperturb}
	\abs{\xi_\nu(z,x)}\le C
	\left(\frac{x}{1+\abs{\sqrt{z}}x}\right)^{\nu+\frac{1}{2}}
	e^{\abs{\im(\sqrt{z})}x},
	\\
\label{eq:A2-kost-sak-tes-unperturb}
	\abs{\xi'_\nu(z,x)}\le C
	\left(\frac{x}{1+\abs{\sqrt{z}}x}\right)^{\nu-\frac{1}{2}}
	e^{\abs{\im(\sqrt{z})}x}.
\end{gather}

Let $H_{\nu,s}$ denote the closure of the minimal operator defined by 
$\tau_\nu$, plus boundary condition
\eqref{eq:boundary-left-condition-limit-circle} whenever $\nu\in[0,1)$,
on the interval $(0,s]$. This operator is regular, symmetric and has
deficiency indices $(1,1)$. As a consequence, 
$\xi_\nu(z,\cdot) \in \ker(H_{\nu,s}^* - z I)$ for all $z\in\C$.

Let $H_{\nu,s,\gamma}$ denote the selfadjoint extension of $H_{\nu,s}$
associated with the boundary condition
$\varphi(s)\cos(\gamma) + \varphi'(s)\sin(\gamma) = 0$, $\gamma\in[0,\pi)$.
For $\gamma=0$ (Dirichlet boundary condition at $x=s$), the spectrum of 
$H_{\nu,s,0}$ is given by the zeros $\{j_{\nu,n}\}_{n=1}^\infty$ of $J_\nu$,
namely,
\begin{equation*}
\sigma(H_{\nu,s,0})= \left\{\left(\frac{j_{\nu,n}}{s}\right)^2\right\}_{n\in\N}.
\end{equation*}
For $\gamma\in(0,\pi)$, the spectrum of $\sigma(H_{\nu,s,\gamma})$ is given 
by the zeros of the real entire function 
$\xi_{\nu}'(z,s)+\xi_{\nu}(z,s)\cot\gamma$. 
The lowest zero of this function, denoted $\lambda_{\nu,0}^\gamma$, has
the same sign as $\nu+1/2+s\cot\gamma$. All the other zeros are positive, thus
we can write
\begin{equation}
\label{eq:unperturbed-spectrum-gamma-not-zero}
\sigma(H_{\nu,s,\gamma})
	= \left\{\left(\frac{j_{\nu,n}^\gamma}{s}\right)^2\right\}_{n\in\N}
		\cup\big\{\lambda_{\nu,0}^\gamma\big\},
\end{equation}
where $\{j_{\nu,n}^\gamma\}_{n=1}^\infty$ are the positive zeros of
\begin{equation} 
\label{eq:condition-one}
wJ_{\nu+1}(w)-\left(\nu+\frac{1}{2}+ s\cot\gamma\right)J_\nu(w).
\end{equation}
We recall the asymptotic formulas (cf. \cite[Eqs.~2.11 and 2.12]{kostenko})
\begin{align}
j_{\nu,n}
&=	\label{eq:estimate-teschl-l}
	\left(n+\frac{2\nu-1}{4}\right)\pi+\mathcal{O}\left(n^{-1}\right),
	\quad n \to \infty,
\\[1mm]
j_{\nu,n}^\gamma
&=	\label{eq:estimate-teschl-2}
	j_{\nu+1,n}+\cO(n^{-1})=
	\left(n+\frac{2\nu+1}{4}\right)\pi+\cO(n^{-1}),
	\quad n \to \infty,
\end{align}
and the asymptotic expansion
\begin{align}
J_\nu(z)
	&=\sqrt{\frac{2}{\pi z}}\left(
		\cos\left(z-\frac{\nu \pi}{2}-\frac{\pi}{4}\right)
		+ e^{\abs{\im (z)}}\cO\left(\abs{z}^{-1}\right)\right),
		\quad \abs{z}\to\infty,\label{eq:asymptotic-bessel-j}
%	\\[1mm]
%Y_\nu(z)
%	&=\sqrt{\frac{2}{\pi z}}\left(
%		\sin\left(z-\frac{\nu \pi}{2}-\frac{\pi}{4}\right)
%		+ e^{\abs{\im (z)}}\cO\left(\abs{z}^{-1}\right)\right),
%		\quad \abs{z}\to\infty.\nonumber%\label{eq:asymptotic-bessel-y}
\end{align}
where the error is uniform in sectors of the form
$\{z\in\C:\abs{z}>r \land \arg(z)\in[-\pi+\delta,\pi-\delta]\}$
\cite[Eq.~10.7.8]{nist}.

Associated with $H_{\nu,s}$, there is the de Branges space
\begin{equation}
\label{eq:associated-dB-space-zero-q}
\cB_{\nu,s} 
	:= \left\{F(z)=\int_0^s \xi_\nu(z,x)\varphi(x)dx : \varphi \in 
	L^2(0,s)\right\},
	\quad
	\norm{F}^2_{\cB_{\nu,s}} = \int_0^s \abs{\varphi(x)}^2dx;
\end{equation}
$E_{\nu,s}(z) \defeq \xi_\nu(z,s) + i \xi_\nu'(z,s)$
is a Hermite-Biehler function that generates $\cB_{\nu,s}$.

Increasing values of the parameter $\nu$ makes $H_{\nu,s}$ more singular; 
a precise meaning of this assertion 
can be found in \cite{kt}. This singular character is also reflected
in the associated de Branges space, in the sense stated next.

\begin{theorem}[cf.~{\cite[Thm.~5.1]{langer-woracek}}, 
{\cite[Thm.~3.1]{siltol2}}]
\label{thm:asocc-N}
Fix $\nu,s\in(0,\infty)$. There exists a real, zero-free entire function 
lying in the set
\begin{equation*}
\assoc_{N(\nu)}\cB_{\nu,s} \defeq \cB_{\nu,s} + z\cB_{\nu,s} 
		+ \cdots + z^{N(\nu)}\cB_{\nu,s}, 
\end{equation*}
where ${N(\nu)} \defeq \min\{n\in\N:n>\frac{\nu+1}{2}\}$; no such a function
exists within $\assoc_{k}\cB_s$ for any $0\le k< N(\nu)$.
\end{theorem}

\subsection{Adding a perturbation}

Given $s\in(0,\infty)$, consider the differential expression %(here $\nu = l + 1/2$)
\begin{equation}
	\label{eq:differential-expression}
	\tau \defeq -\frac{d^2}{dx^2}+
			\frac{\nu^2-1/4}{x^2}+q(x), \quad x\in(0,s), 
			\quad \nu \in [0,\infty).
\end{equation}
We assume that $q\in L^1_{\text{loc}}(0,s)$ is a real-valued function such that
$\widetilde{q}\in L^1(0,s)$, where 
\begin{equation}
\label{eq:conditions-q(x)}
\widetilde{q}(x)
	\defeq \begin{cases}
			x q(x) 						 & \text{if } \nu > 0,\\
			x \left(1-\log(x)\right)q(x) & \text{if } \nu = 0.
\end{cases}
\end{equation}

As shown in \cite[Thm. 2.4]{kostenko}, $\tau$ is regular at $x=s$ whereas at $x=0$ 
it is in the limit point case if $\nu\geq 1$ or in the limit circle case 
if $\nu\in [0,1)$. 

The expression \eqref{eq:differential-expression}, along with the boundary 
condition \eqref{eq:boundary-left-condition-limit-circle} when $\nu\in [0,1)$,
originates a closed, regular, symmetric operator whose deficiency indices are both
equal to $1$ \cite[Sect. 4]{siltol2}. We denote this operator by $H_s$
(the symbol $H_{\nu,q,s}$ would be more accurate but also clumsier).

The corresponding one-parameter family of selfadjoint extensions $H_{s,\gamma}$
($0\le \gamma < \pi$) is determined by the usual boundary condition at $x=s$,
\begin{equation}
\label{eq:domain-selfadjoint-extensions}
D(H_{s,\gamma})
:=\left\{\varphi\in L^2(0,s): \varphi,\varphi'\in\text{AC}(0,s),
\tau\varphi\in L^2(0,s),
\varphi(s)\cos{\gamma}=-\varphi'(s)\sin{\gamma}\right\}
\end{equation}
plus the boundary condition \eqref{eq:boundary-left-condition-limit-circle}
when $\nu\in [0,1)$. The spectrum of $H_{s,\gamma}$ is purely discrete, of
multiplicity one, with at most a finite number of negative eigenvalues 
\cite[Thm. 2.4]{kostenko}.

We henceforth assume $\nu>0$ (see Remark~\ref{rem:nu-not-zero}).
By \cite[Lemma 2.2]{kostenko}, the eigenvalue equation 
$\tau\varphi = z\varphi$ ($z\in\C$) admits a solution $\xi(z,x)$, 
real entire with respect to $z$, with derivative $\xi'(z,x)$ also
real entire, that satisfy the estimates
\begin{equation}
\label{eq:estimate-delta-xi}
\abs{\xi(z,x)-\xi_\nu(z,x)}
	\le C \left(\frac{x}{1+\abs{\sqrt{z}}x}\right)^{\nu+\frac12}
		e^{\abs{\im(\sqrt{z})}x}\int_0^x\frac{u\abs{q(u)}}{1+\abs{\sqrt{z}}u}du,
\end{equation}
and
\begin{equation}
\label{eq:estimate-delta-xi-prime}
\abs{\xi'(z,x)-\xi_\nu'(z,x)}
	\le C \left(\frac{x}{1+\abs{\sqrt{z}}x}\right)^{\nu-\frac12}
		e^{\abs{\im(\sqrt{z})}x}\int_0^x\frac{u\abs{q(u)}}{1+\abs{\sqrt{z}}u}du,
\end{equation}
for some constant $C=C(\nu,q,s)$, with $\nu,s\in(0,\infty)$, so the bounds above
are uniform for $x\in(0,s]$.
Note that \eqref{eq:A1-kost-sak-tes-unperturb} and \eqref{eq:estimate-delta-xi} 
(respectively, \eqref{eq:A2-kost-sak-tes-unperturb} and
\eqref{eq:estimate-delta-xi-prime}) imply
\begin{gather}
\label{eq:another-estimate}
\abs{\xi(z,x)}
	\le C \left(\frac{x}{1+\abs{\sqrt{z}}x}\right)^{\nu+\frac12}
			e^{\abs{\im(\sqrt{z})}x}
			\left(1+\norm{\widetilde{q}}_{L^1(0,s)}\right),
	\\[1mm]
\label{eq:another-estimate-prime}
\abs{\xi'(z,x)}
	\le C \left(\frac{x}{1+\abs{\sqrt{z}}x}\right)^{\nu-\frac12}
			e^{\abs{\im(\sqrt{z})}x}
			\left(1+\norm{\widetilde{q}}_{L^1(0,s)}\right),
\end{gather}
for all $x\in(0,s]$.

\begin{lemma}[cf. {\cite[Lemma~4.1]{siltol2}}]
\label{lem:norm-estimate-fund-solution}
Assume $\widetilde{q}\in L^r(0,s)$ with $r\in[1,\infty]$. Then,
\begin{equation*}
\norm{\xi(w^2,\cdot)-\xi_\nu(w^2,\cdot)}_{L^2(0,s)}
	= e^{\abs{\im(w)}}\times
	\begin{cases}
		o(\abs{w}^{-\nu - \frac12})              & \text{if } r=1,
		\\[1mm]
		\cO(\abs{w}^{-\nu - \frac12 - \frac1p})      & \text{if } 1<r<\infty,
		\\[1mm]
		\cO(\abs{w}^{-\nu - \frac32}\log\abs{w}) & \text{if } r=\infty,
	  \end{cases}
\end{equation*}
as $w\to\infty$, where $p$ obeys $1/p + 1/r = 1$.
\end{lemma}

\begin{remark}
\label{rem:refined-formula}
The proof of \cite[Lemma 2.2]{kostenko} leads to a more refined decomposition,
namely,
\begin{equation}
\label{eq:xi-refined}
\xi(z,x) = \xi_\nu(z,x) + \xi_{\nu,1}(z,x) + \Xi(z,x)
\end{equation}
for all $z\in\C$ and $x\in(0,s]$, where
\begin{equation}
\label{eq:asymptotic-formula-xi-refined}
\xi_{\nu,1}(z,x)
	= \xi_\nu(z,x)\int_0^x q(y)\theta_\nu(z,y)\xi_\nu(z,y)dy
		- \theta_\nu(z,x)\int_0^x q(y)\left(\xi_\nu(z,y)\right)^2 dy
\end{equation}
obeys
\begin{equation}
\label{eq:estimate-xi-one}
\abs{\xi_{\nu,1}(z,x)}
	\le C \left(\frac{x}{1+\abs{\sqrt{z}}x}\right)^{\nu+\frac12}
			e^{\abs{\im(\sqrt{z})}x}
			\int_0^x\frac{u\abs{q(u)}}{1+\abs{\sqrt{z}}u}du,
\end{equation}
and
\begin{equation}
\label{eq:res-term-refined}
\abs{\Xi(z,x)}
	\le C \left(\frac{x}{1+\abs{\sqrt{z}}x}\right)^{\nu+\frac12}
		e^{\abs{\im(\sqrt{z})}x}
		\left(\int_0^x\frac{u\abs{q(u)}}{1+\abs{\sqrt{z}}u}du\right)^2.
\end{equation}
\end{remark}

From \eqref{eq:res-term-refined} one immediately obtains the following estimate:
\begin{lemma}
\label{lem:some-tools}
Fix $\nu,s\in(0,\infty)$. Suppose $\widetilde{q}\in L^r(0,s)$ with $r\in(2,\infty]$.
Then, there exists positive constants $C$ and $\delta>0$ such that
\begin{equation*}
\norm{\Xi(t^2,\cdot)}_{L^2(0,s)}
	\le C t^{-\nu- \frac32 -\delta},
\end{equation*}
for all $t>0$.
\end{lemma}
%\begin{proof}
%Trivial \fixme.
%\end{proof}

From \eqref{eq:estimate-delta-xi} and \eqref{eq:estimate-delta-xi-prime},
it follows that $\xi(z,x)$ also obeys boundary condition 
\eqref{eq:boundary-left-condition-limit-circle} whenever $\nu\in(0,1)$.
This in turn implies
$\xi(z,\cdot) \in \ker(H_s^* - z I)$.
Hence, in view of \eqref{eq:domain-selfadjoint-extensions}, the spectrum of 
$H_{s,\gamma}$ is given by
\begin{equation}
\label{eq:spectrum-selfadjoint-operators}
\sigma(H_{s,\gamma})
	= \{\lambda\in\R : \xi(\lambda,s)\cos\gamma + \xi'(\lambda,s)\sin\gamma=0\}.
\end{equation}
In particular, if $\lambda\in\sigma(H_{s,\gamma})$, then $\xi(\lambda,\cdot)$ 
is the corresponding eigenfunction (up to normalization).

%In the sequel we shall need the following estimates:

Arrange the elements of $\sigma(H_{s,\gamma})$ according to 
increasing values. Let us denote (and enumerate) them as follows,
\begin{equation*}
\sigma(H_{s,\gamma}) 
	= \begin{cases}
		\{t_n^2\}_{n=1}^\infty & \text{if }\gamma=0,
		\\[2mm]
		\{t_n^2\}_{n=0}^\infty & \text{if }\gamma\neq 0
	  \end{cases}
\end{equation*}
(the finitely many negative eigenvalues have imaginary values of $t_n$). 
According to \cite[Thm.~2.5]{kostenko},
\begin{equation}
\label{eq:asymptotic-eigenvalues}
t_n = \begin{cases}
		\left(n+\frac{2\nu-1}{4}\right)\frac{\pi}{s} 
				+ \epsilon_n + \cO(n^{-1})        & \text{if }\gamma=0,
		\\[2mm]
		\left(n+\frac{2\nu+1}{4}\right)\frac{\pi}{s} 
				+ \epsilon_n^\gamma + \cO(n^{-1}) & \text{if }\gamma\neq 0,
	  \end{cases}
	\quad n\to\infty,
\end{equation}
where 
\begin{equation*}
\epsilon_n        = \cO\left(\int_0^s\frac{x\abs{q(x)}}{s+n\pi x}dx\right),
\quad
\epsilon_n^\gamma = \cO\left(\int_0^s\frac{x\abs{q(x)}}{s+n\pi x}dx\right),
\quad n\to\infty.
\end{equation*}
Assuming $\widetilde{q}\in L^r(0,s)$ with $r\in[1,\infty]$, it follows that
\begin{equation}
\label{eq:asym-residual-term}
\int_0^s\frac{x\abs{q(x)}}{s+n\pi x}dx 
	=\begin{cases}
		o(1)               & \text{if } r = 1,
		\\
		\cO(n^{-1+\frac1r})    & \text{if } 1<r<\infty,
		\\
		\cO(n^{-1}\log(n)) & \text{if } r = \infty,
	 \end{cases}
\end{equation}
as $n\to\infty$.

\begin{lemma}
\label{lem:denominator}
Suppose $\widetilde{q}\in L^1(0,s)$. Then, there exist $0<C<D<\infty$ and 
$n_0\in\N$ such that
\[
C n^{-\nu-\frac12} 
	\le \norm{\xi(t_n^2,\cdot)}_{L^2(0,s)}
	\le D n^{-\nu-\frac12}
\]
for all $n\ge n_0$.
\end{lemma}
\begin{proof}
We have the obvious inequalities
\[
\abs{\norm{\xi_\nu(t_n^2)}-\norm{\xi(t_n^2)-\xi_\nu(t_n^2)}}
	\le \norm{\xi(t_n^2)}
	\le \norm{\xi_\nu(t_n^2)}+\norm{\xi(t_n^2)-\xi_\nu(t_n^2)}
\]
In what follows we assume $n$ so large that $t_n^2$ is positive. Then, on one 
hand,
\begin{align*}
\norm{\xi_\nu(t_n^2)}^2
	&= \frac{\pi}{2} t_n^{-2\nu}\int_0^s x\left(J_\nu(t_n x)\right)^2 dx
	\\
	&= \frac{\pi}{4} s^2 t_n^{-2\nu}
		\left(\left(J_\nu(t_ns)\right)^2 - J_{\nu-1}(t_ns)J_{\nu+1}(t_ns)\right).
\end{align*}
Combining \eqref{eq:asymptotic-bessel-j}, 
\eqref{eq:asymptotic-eigenvalues} and \eqref{eq:asym-residual-term}, one can see
that
\begin{equation*}
\left(J_\nu(t_ns)\right)^2 - J_{\nu-1}(t_ns)J_{\nu+1}(t_ns)
	= \frac{2}{\pi s t_n} (1 + o(1)),\quad n\to\infty.
\end{equation*}
That is,
\begin{equation}
\label{eq:boredom-1}
\norm{\xi_\nu(t_n^2)} = \sqrt{\frac{s}{2}} t_n^{-\nu-\frac12}(1 + o(1)),\quad 
n\to\infty.
\end{equation}
On the other hand, Lemma~\ref{lem:norm-estimate-fund-solution} implies
\begin{equation}
\label{eq:boredom-2}
\norm{\xi(t_n^2)-\xi_\nu(t_n^2)}
	= o\left(t_n^{-\nu - \frac12}\right).
\end{equation}
Finally, the assertion follows after combining 
\eqref{eq:asymptotic-eigenvalues}, 
\eqref{eq:boredom-1} and \eqref{eq:boredom-2}.
\end{proof}

Associated with the symmetric operator $H_{s}$, one has the de Branges space
\begin{equation}
\label{eq:associated-dB-space}
\cB_{s} 
	:= \left\{F(z)=\int_0^s \xi(z,x)\varphi(x)dx : \varphi \in 
	L^2(0,s)\right\},\quad
\norm{F}^2_{\cB_{s}} 
		= \int_0^s \abs{\varphi(x)}^2dx.
\end{equation}
The corresponding reproducing kernel is 
\[
K_s(z,w)=\inner{\xi(\cc{z},\cdot)}{\xi(\cc{w},\cdot)}_{L^2(0,s)}.
\]

\begin{theorem}[cf. {\cite[Thm.~4.2]{siltol2}}]
\label{thm:invariance-spaces}
Fix $\nu,s\in(0,\infty)$.
Assume $\widetilde{q}\in L^r(0,s)$ with $r\in(2,\infty]$. 
Then $\cB_{s}=\cB_{\nu,s}$ setwise. Consequently, $\assoc_{N(\nu)}\cB_s$
contains a zero-free real entire function but no such a function lies
in $\assoc_{k}\cB_s$ for any $0\le k< N(\nu)$.
\end{theorem}

\begin{remark}
\label{rem:unitary-invariants}
Let $\Phi: L^2(0,s)\to\cB_{s}$ be the unitary operator
defined by the rule $F(z)=(\Phi\varphi)(z)$.
Then $S_{s} = \Phi H_{s}\Phi^{-1}$, where $S_{s}$ is 
the operator of multiplication by the independent variable in $\cB_{s}$, 
Moreover, the corresponding selfadjoint extensions are analogously related,
viz., $S_{s,\gamma} = \Phi H_{s,\gamma}\Phi^{-1}$. 
Thus, when referring to unitary invariants (such as the spectrum),
we use interchangeably either $S_{s,\gamma}$ or $H_{s,\gamma}$ throughout this text.
\end{remark}

\section{Oversampling}
\label{sec:oversampling}

%\subsection{A sufficient condition}

As mentioned in the Introduction, a discussion of the oversampling 
property in de Branges spaces involves certain weighted $\ell_p$ spaces.
For the class of de Branges spaces treated in this paper, we need the following
ones: Given $1\le p \le \infty$, define
\begin{equation}
\label{eq:ell-simple}
\ell_p(\nu)
	\defeq\left\{\{\beta_n\}_{n\in\N}\subset\C:
		\{\beta_n n^{-\nu-\frac12}\}_{n\in\N}\in\ell_p\right\},
\end{equation}
and
\begin{equation*}
%\label{eq:ell-infty-weighted-space}
\ell_{p}(\nu,s,q,\gamma)\defeq
\left\{ \{\beta_n\}_{n\in\N}\subset\C 
	: \left\{\beta_n K_s(\lambda_n,\lambda_n)^{-\frac12}\right\}_{n\in\N}\in\ell_p
\right\},
\end{equation*}
where $\{\lambda_n\}$ is the (ordered) spectrum of $H_{s,\gamma}$.
%The cases of interest in this work are $p=2,\infty$.

\begin{corollary}
\label{cor:spaces-lp}
Fix $\nu,s\in(0,\infty)$ and $\gamma\in[0,\pi)$. Suppose 
$\widetilde{q}\in L^1(0,s)$. Then, $\ell_{p}(\nu,s,q,\gamma)=\ell_{p}(\nu)$ 
setwise. 
\end{corollary}
\begin{proof}
Use Lemma~\ref{lem:denominator}.
\end{proof}

For every $\varphi\in L^2(0,s)$, one clearly has
\begin{equation}
\label{eq:decomposition-function-L2}
\varphi(x) 
	= \sum_{\lambda_n\in\sigma(H_{s,\gamma})}
		\frac{1}{K_s(\lambda_n,\lambda_n)}\inner{\xi(\lambda_n,\cdot)}
		{\varphi(\cdot)}_{L^2(0,s)}
		\xi(\lambda_n,x),\quad \text{a.e.\ } x\in(0,s],
\end{equation}
where the convergence takes place with respect to the $L^2$-norm.
Hence the sampling formula
\begin{equation}
\label{1ra formula kramer}
F(z)=\sum_{\lambda_n\in\sigma(S_{s,\gamma})}\frac{K_s(z,\lambda_n)}
		{K_s(\lambda_n,\lambda_n)}\, F(\lambda_n),
		\quad F\in\cB_s,
\end{equation}
holds true, where the convergence is with respect to the $\cB_s$-norm, 
which in turn implies uniform convergence in compact subsets of $\C$ 
\cite[Prop. 1]{siltol1}.
Since
\begin{equation}
\label{eq:using-parseval}
\norm{F}^2_{\cB_s}=\sum_{\lambda_n\in\sigma(S_{s,\gamma})} 
\frac{\abs{F(\lambda_n)}^2}{K_s(\lambda_n,\lambda_n)},
\quad F \in \cB_s,
\end{equation}
and taking into account Corollary~\ref{cor:spaces-lp},
the sequence $\{F(\lambda_n) : \lambda_n\in\sigma(S_{s,\gamma})\}$
belongs to ${\ell}_2(\nu)$. Clearly, if one substitutes
$F(\lambda_n)$ by $F(\lambda_n)+\delta_n$ with 
$\delta=\{\delta_n\}\in \ell_2(\nu)$, then
\eqref{1ra formula kramer} produces an approximation $F_\delta\in\cB_s$
that satisfies 
\[
\abs{F_\delta(z)-F(z)}\le C(\K)\norm{\delta}_{\ell_2(\nu)}
\] 
uniformly for $z$ in any given compact subset $\K\subset\C$.

Fix $0<a<b<\infty$.
Any $\varphi\in L^2(0,a)$ can be regarded as an element of $L^2(0,b)$ since
$\varphi = \varphi \chi_{(0,a]} + 0 \chi_{(a,b]}$, where $\chi_E$ denotes 
the characteristic function of a set $E$. Define
\begin{equation}
\label{eq:function-R}
\cR_{ab}(x) \defeq  \chi_{(0,a]}(x) + \frac{b-x}{b-a}\chi_{(a,b]}(x).
\end{equation}
In this way, $\varphi=\varphi\cR_{ab}$ for all $\varphi\in L^2(0,a)$. 
Hence, using \eqref{eq:decomposition-function-L2} with $s=b$,
\begin{equation}     
\label{eq:f-product-R}
\varphi(x) 
	= \sum_{\lambda_n \in 
	\sigma(H_{b,\gamma})}\frac{1}{K_b(\lambda_n,\lambda_n)}
	\inner{\xi(\lambda_n,\cdot)}
	{\varphi(\cdot)}_{L_2(0,b)}\cR_{ab}(x)\xi(\lambda_n,x),
	\quad \text{a.e.\ } x\in(0,b],
\end{equation}
where the series converges with respect to the norm of $L^2(0,b)$.
Consider
\begin{equation*}
F(z)=\inner{\xi(\cc{z},\cdot)}{\varphi(\cdot)}_{L^2(0,b)},
\quad z \in \C.
\end{equation*}
Plugging \eqref{eq:f-product-R} in the previous equation, we 
arrive at
\begin{equation}     \label{eq:F-with-modified-kernel}
F(z) = \sum_{\lambda_n\in \sigma(S_{b,\gamma})}
		\frac{1}{K_b(\lambda_n,\lambda_n)}
		\inner{\xi(\cc{z},\cdot)}{\cR_{ab}(\cdot)\xi(\lambda_n,\cdot)}_{L^2(0,b)}
		F(\lambda_n),
\end{equation}
which converges uniformly in compact subsets of $\C$.

Define
\begin{equation}
\label{eq:function-Jab}
J_{ab}(z,w)\defeq
	\inner{\xi(\cc{z},\cdot)}{\cR_{ab}(\cdot)\xi(\cc{w},\cdot)}_{L^2(0,b)}.
\end{equation}
Note that $J_{ab}(\cdot,w)\in\cB_b$ for every $w\in\C$, and 
$J_{ab}(w,z)=\cc{J_{ab}(z,w)}$. % among other properties \fixme. 
We will prove that the spaces $\cB_s$ satisfy the following condition:

\begin{enumerate}[label={\bf(sc\arabic*)},ref={(sc\arabic*)},series=suf_cond,leftmargin=*]
\item \label{it:condition-oversampling}
	\label{hyp:convergence-oversampling} Given $0<a<b$ and any selfadjoint
	extension $S_{b,\gamma}$ of $S_b$, the series
	\begin{equation*}
	\sum_{\lambda_n \in \sigma(S_{b,\gamma})}
	\frac{\abs{J_{ab}(z,\lambda_n)}}{\sqrt{K_b(\lambda_n,\lambda_n)}}
	\end{equation*}
	converges uniformly in compact subsets of $\C$.
\end{enumerate}

\begin{proposition}
\label{prop:condition-oversampling-general}
Fix $\nu,b\in(0,\infty)$, $a\in(0,b)$, and $\gamma\in[0,\pi)$.
Suppose $q\in\text{AC}_\text{loc}(0,b]$ such that
$\widetilde{q}\in L^r(0,b)$ with $r\in(2,\infty]$.
Then, $\cB_{b}$ satisfies \ref{hyp:convergence-oversampling}.
%
%Given a compact set $\K\subset\C$, there exists $C=C(\K, q, \gamma, a, b)$ 
%such that
%\[
%\abs{\inner{\xi(\cc{z},\cdot)}{R_{ab}(\cdot)\xi(t_n^2,\cdot)}_{L^2(0,b)}}
%%	
%%-\inner{\xi_{\nu}(\cc{z},\cdot)}{R_{ab}(\cdot)\xi_{\nu}(s_n^2,\cdot)}_{L^2(0,b)}}
%	\le C n^{-\nu-2},
%\]
%for all $z\in\K$.
\end{proposition}
\begin{proof}
Denote $\sigma(H_{b,\gamma})=\{t_n^2\}$.
In view of Lemma~\ref{lem:denominator}, it suffices to show that, given a compact
subset $\K\subset\C$, there exist $n_0\in\N$, positive constant 
$C=C(\K,\nu,q,\gamma,a,b)$ and $\delta>0$ such that
\[
\abs{\inner{\xi(\cc{z})}{R_{ab}\,\xi(t_n^2)}_{L^2(0,b)}}\le C n^{-\nu-\frac32-\delta},
\]
for all $z\in\K$ and $n\ge n_0$. For the purpose of this proof
$\inner{\cdot}{\cdot}:=\inner{\cdot}{\cdot}_{L^2(0,b)}$. Resorting to 
\eqref{eq:xi-refined}, one can write
\begin{equation}
\label{eq:be-careful}
\inner{\xi(\cc{z})}{R\,\xi(t^2)} 
	= \inner{\xi(\cc{z})}{R\,\xi_{\nu}(t^2)}
		+ \inner{\xi(\cc{z})}{R\,\xi_{\nu,1}(t^2)} 
		+ \inner{\xi(\cc{z})}{R\,\Xi(t^2)}
\end{equation}
where we have abbreviated $R:=R_{ab}$ and $t:=t_n$.

An integration by parts (see \ref{app:first-computation}) reduces the 
first term in \eqref{eq:be-careful} to
\begin{align}
\inner{\xi({z})}{R\,\xi_\nu(t^2)}
	&= -\,\frac{1}{t^2}\left.\frac{1}{b-a}
		\left(\xi(z,b)\xi_\nu(t^2,b) - \xi(z,a)\xi_\nu(t^2,a)\right)
		\right.
	\nonumber
%	\\
%	&\qquad +\left. 2\int_a^b \xi'(z,x)\xi_\nu(t^2,x)dx\right.
%	\nonumber
	\\
	&\qquad -\frac{1}{t^2}\left.\int_0^bR(x)(q(x)-z)\xi(z,x)\xi_\nu(t^2,x)dx\right.
	\nonumber
	\\[1mm]
	&\qquad - \left.\frac{2}{t^2}
		\int_a^b\frac{1}{b-a}\xi'(z,x)\xi_\nu(t^2,x)dx\right..
	\label{eq:one}
\end{align}
Using \eqref{eq:A1-kost-sak-tes-unperturb}, \eqref{eq:another-estimate},
the fact that $\xi(z,x)$ and $\xi'(z,x)$ are entire with respect to $z$ 
for every $x\in(0,\infty)$, and noting that $\widetilde{q}\in L^1(0,b)$ due to
our hypotheses, 
one can see that \eqref{eq:one} implies
\begin{equation*}
\abs{\inner{\xi(\cc{z})}{R\,\xi_\nu(t_n^2)}}
	\le C_1 n^{-\nu-\frac52},\quad z\in\K,\quad n\ge n_0,
\end{equation*}
for some $C_1=C_1(\K,\nu,q,\gamma,a,b)>0$.

The second term in \eqref{eq:be-careful} is computed in \ref{app:third-computation}
(assuming $q$ locally absolutely continuous), the result being
\begin{align}
\inner{\xi(\cc{z})}{R\,\xi_{\nu,1}(t^2)}
	&= -\, \frac{1}{t^2}\left.\frac{1}{b-a}
			\left(\xi(z,b)\xi_{\nu,1}(t^2,b)-\xi(z,a)\xi_{\nu,1}(t^2,a)\right)\right.
	\nonumber
	\\[1mm]
	&\qquad - \frac{1}{t^2}
			\int_0^b R(x)q(x)\xi(z,x)\xi_\nu(t^2,x)dx
	\nonumber
	\\[1mm]
	&\qquad - \frac{1}{t^2}
			\int_0^b R(x)(q(x)-z)\xi(z,x)\xi_{\nu,1}(t^2,x)dx
	\nonumber
	\\[1mm]
	&\qquad + \frac{2}{t^2}
			\int_a^b\frac{1}{b-a}\xi'(z,x)\xi_{\nu,1}(t^2,x)dx.
	\label{eq:almost-done}
\end{align}
The estimates \eqref{eq:another-estimate} and \eqref{eq:estimate-xi-one}
imply
\begin{equation*}
\abs{\xi(z,s)\xi_{\nu,1}(t^2,s)}
	\le C_2 t^{-\nu-\frac12}
\end{equation*}
for some $C_2=C_2(\K,\nu,q,s)>0$; this bound takes care of the first
two terms in \eqref{eq:almost-done}. 
Also,
\begin{equation}
\label{eq:for-aliasing}
\int_0^b\abs{q(x)\xi(z,x)\xi_{\nu}(t^2,x)} dx
	\le C_2 \int_0^b \abs{xq(x)} 
		\frac{x^\nu e^{\abs{\im(\sqrt{z})}x}}
		{\left(1+\abs{\sqrt{z}}x\right)^{\nu+\frac12}}
		\frac{x^\nu}{(1+tx)^{\nu+\frac12}}dx
	\le C_2 t^{-\nu},
\end{equation}
where $C_2=C_2(\K,\nu,q,b)>0$; here we have used \eqref{eq:A1-kost-sak-tes-unperturb}
and \eqref{eq:another-estimate}, along with fact our hypothesis on $q$ implies 
$\widetilde{q}\in L^1(0,b)$. A similar argument shows
\begin{equation}
\label{eq:for-aliasing-too}
\int_0^b\abs{(q(x)-z)\xi(z,x)\xi_{\nu,1}(t^2,x)} dx
	\le C_2 t^{-\nu}
\end{equation}
and
\begin{equation*}
\int_a^b\abs{\xi'(z,x)\xi_{\nu,1}(t^2,x)} dx
	\le C_2 t^{-\nu-\frac12},
\end{equation*}
where the latter is due to \eqref{eq:another-estimate-prime} and
\eqref{eq:estimate-xi-one}. Therefore, taking into account 
\eqref{eq:asymptotic-eigenvalues}, one has
\begin{equation*}
\abs{\inner{\xi(\cc{z})}{R\,\xi_{\nu,1}(t^2)}}
	\le C_2 n^{-\nu -2},\quad z\in\K,\quad n\ge n_0,
\end{equation*}
for some $C_2=C_2(\K,\nu,q,\gamma,a,b)>0$.

Finally, with the help of Lemma~\ref{lem:some-tools}, the third term in 
\eqref{eq:be-careful} admits the bound
\begin{equation*}
\abs{\inner{\xi(\cc{z})}{R\,\Xi(t_n^2)}}
	\le \norm{\xi(\cc{z})}_{L^2(0,b)}\norm{\Xi(t_n^2)}_{L^2(0,b)}
	\le C_3 n^{-\nu- 3/2 -\delta},
\end{equation*}
with $C_3=C_3(\K,\nu,q,\gamma,b)>0$ and some $\delta>0$ (that depends on $r>2$ from
our hypothesis on $q$).
\end{proof}

\begin{theorem}
Fix $\nu,b\in(0,\infty)$, $a\in(0,b)$ and $\gamma\in[0,\pi)$. 
Assume that $q$ is a real-valued 
function belonging to $\text{AC}_\text{loc}(0,b]$ such that 
$\widetilde{q}\in L^r(0,b)$ for some $r\in(2,\infty]$.
Given $\epsilon=\{\epsilon_n\}\in\ell_\infty(\nu)$ and $F\in\cB_a$,
define
\begin{equation*}
F_\epsilon(z)=\sum_{\lambda_n\in\sigma(S_{b,\gamma})}
		\frac{J_{ab}(z,\lambda_n)}
		{K_b(\lambda_n,\lambda_n)}\left(F(\lambda_n)+\epsilon_n\right).
\end{equation*}
Then, for every compact set $\K\subset\C$, there exists $C(\K)=C(\K,\nu,q,\gamma,a,b)>0$ 
such that
\[
\abs{F(z)-F_\epsilon(z)} \le C(\K)\norm{\epsilon}_{\ell_\infty(\nu)},
	\quad z\in \K,
\]
uniformly for all $F\in\cB_a$.
\end{theorem}
\begin{proof}
It is a straightforward consequence of 
Proposition~\ref{prop:condition-oversampling-general} combined with
Corollary~\ref{cor:spaces-lp}.
\end{proof}

\begin{remark}
A closer inspection to the estimates on \eqref{eq:almost-done} reveals that 
\[
C(\K,\nu,q,\gamma,a,b)=\cO((b-a)^{-1}),\quad a\to b.
\]
This fact, already known for Paley-Wiener spaces (cf. \eqref{eq:oversampling-PW}), is 
somewhat expected since the stability of the oversampling formula depends
on $\cB_a$ being a proper de Branges subspace of $\cB_b$.
\end{remark}

\section{Aliasing}
\label{sec:aliasing}

%\subsection{Another sufficient condition}

A de Branges space $\cB_b$ has the aliasing (or undersampling) property if,
given any de Branges subspace $\cB_a\subsetneq\cB_b$, there exists
a selfadjoint extension $S_{a,*}$ of $S_a$ such that the series
\begin{equation*}
%\label{def:aliasing}
\sum_{\lambda_n\in\sigma(S_{a,*})}
	\frac{K_a(z,\lambda_n)}{K_a(\lambda_n,\lambda_n)}\, F(\lambda_n),
\end{equation*}
converges absolutely in compact subsets of $\C$, for every function
$F\in\cB_b\setminus\cB_a$.

%\begin{todo}
%It may be interesting to give some explanation for the word ``aliasing''.
%\end{todo}

Suppose $\cB$ has the aliasing property. Then, for every 
$F\in\cB_b\setminus\cB_a$
one can define
\begin{equation*}
\tilde{F}(z) =
\sum_{\lambda_n\in\sigma(S_{a,*})}
	\frac{K_a(z,\lambda_n)}{K_a(\lambda_n,\lambda_n)}\, F(\lambda_n).
\end{equation*}
We expect $\tilde{F}$ be an approximation to $F$ obtained from samples
that are more ``sparse'' than those required for the sampling
formula \eqref{1ra formula kramer}. This vaguely worded claim can be made 
precise
for the class of de Branges spaces under consideration. With this purpose in 
mind, we 
formulate a suitable sufficient condition for aliasing.

\begin{enumerate}[resume*=suf_cond]
\item\label{hyp:convergence-undersampling} Given $0<a<b$,
	there exists $\gamma\in[0,\pi)$ such that the series
	\begin{equation*}
	\sum_{\lambda_n \in \sigma(H_{a,\gamma})}
	\frac{K_a(z,\lambda_n)}{K_a(\lambda_n,\lambda_n)}
	\xi(\lambda_n,x)
	\end{equation*}
	converges absolutely and uniformly for $(z,x)\in\K\times[0,b]$, 
	where $\K$ is any compact subset of $\C$.
\end{enumerate}

The bulk of this section consists of showing
that the de Branges spaces discussed in this work satisfy 
\ref{hyp:convergence-undersampling}. 
We start by defining
\begin{equation*}
Q_1(z,x) \defeq \int_0^x q(y)\theta_\nu(z,y)\xi_\nu(z,y) dy,
\quad
Q_2(z,x) \defeq \int_0^x q(y)\left(\xi_\nu(z,y)\right)^2 dy.
\end{equation*}
Clearly \eqref{eq:asymptotic-formula-xi-refined} becomes
\begin{equation}
\label{eq:the-other-xi}
\xi_{\nu,1}(z,x) 
	= \xi_\nu(z,x)Q_1(z,x) - \theta_\nu(z,x)Q_2(z,x).
\end{equation}
Also,
\begin{equation*}
Q_1(t^2,a)\xi_{\nu+1}(t^2,a) - t^{-2}Q_2(t^2,a)\theta_{\nu+1}(t^2,a)
= \int_0^a H_\nu(t^2,a,y)q(y)\xi_\nu(t^2,y)dy,
\end{equation*}
where
\begin{equation*}
H_\nu(t^2,x,y) \defeq
\xi_{\nu+1}(t^2,x)\theta_{\nu}(t^2,y) - t^{-2}\xi_{\nu}(t^2,y)\theta_{\nu+1}(t^2,x).
\end{equation*}
From \eqref{eq:fundamental-solution-unperturbed-good} and 
\eqref{eq:fundamental-solution-unperturbed-bad}, one can verify that
\begin{equation*}
H_\nu(t^2,x,y) = \frac{\pi}{2}t^{-1}\sqrt{xy}
\left(J_{\nu+1}(tx) Y_{\nu}(ty) - J_{\nu}(ty) Y_{\nu+1}(tx)\right).
\end{equation*}

\begin{lemma}
\label{lem:relief}
Suppose $\widetilde{q}\in L^r(0,a)$ with $r\in(2,\infty]$. There exist $C>0$ and
$\delta>0$ such that
\begin{equation*}
\abs{Q_1(t^2,a)\xi_{\nu+1}(t^2,a) - t^{-2}Q_2(t^2,a)\theta_{\nu+1}(t^2,a)}
	\le C t^{-\nu-\frac32 - \delta}
\end{equation*}
for all $t\ge 1$.
\end{lemma}
\begin{proof}
Resorting to an argument like in the proof of \cite[Lemma~A.1]{kostenko}, one
can prove (assuming $t\in\R$)
\begin{equation*}
\abs{H_\nu(t^2,x,y)}
\le C t^{-2}\left(\left(\frac{tx}{1+tx}\right)^{\nu+\frac32}
	\left(\frac{1+ty}{ty}\right)^{\nu-\frac12}
	+ \left(\frac{ty}{1+ty}\right)^{\nu+\frac12}
	\left(\frac{1+tx}{tx}\right)^{\nu+\frac12}\right).
\end{equation*}
Noting that the function $f(x)=x(1+x)^{-1}$ ($x\in\R_+$) is increasing and bounded, and
recalling \eqref{eq:A1-kost-sak-tes-unperturb}, it follows that
\begin{align*}
\int_0^a &\abs{H_\nu(t^2,a,y)q(y)\xi_\nu(t^2,y)}dy
	\\
	&\le C t^{-\nu-\frac52}\left(\int_0^a \frac{ty}{1+ty}\abs{q(y)}dy
	+ \left(\frac{1+ta}{ta}\right)^{\nu+\frac12}
	\int_0^a \left(\frac{ty}{1+ty}\right)^{2\nu+1}\abs{q(y)}dy\right).
\end{align*}
Assuming $t\ge 1$, it reduces to
\begin{equation*}
\int_0^a \abs{H_\nu(t^2,a,y)q(y)\xi_\nu(t^2,y)}dy
	\le C t^{-\nu-\frac52}\int_0^a \frac{ty}{1+ty}\abs{q(y)}dy.
\end{equation*}
Suppose $r\in(2,\infty)$. Then,
\begin{equation*}
\int_0^a \frac{ty}{1+ty}\abs{q(y)}dy
	\le t\left(\int_0^a (1+ty)^{-p}dy\right)^{\frac1p}\norm{\widetilde{q}}_{L^r(0,a)}
	\le C t^{\frac1r}.
\end{equation*}
Therefore,
\begin{equation*}
\int_0^a \abs{H_\nu(t^2,a,y)q(y)\xi_\nu(t^2,y)}dy
	\le C t^{-\nu-\frac32 - \frac1p}.
\end{equation*}
The argument for $r=\infty$ is analogous hence omitted.
\end{proof}

\begin{proposition}
\label{prop:condition-aliasing-general}
Given $\nu\in(0,\infty)$ and $0<a<b<\infty$,
\ref{hyp:convergence-undersampling} is satisfied for all $\gamma\in(0,\pi)$ 
whenever $q\in\text{AC}_\text{loc}(0,b]$ in addition to 
$\widetilde{q}\in L^r(0,b)$ for some $r\in(2,\infty]$.
\end{proposition}
\begin{proof}
We show the statement assuming $r\in(2,\infty)$; the
proof for the remaining case is similar save for some minor differences. 
Choose any $\gamma\in(0,\pi)$ and denote 
$\sigma(H_{a,\gamma})=\{t_n^2\}_{n=0}^\infty$.
Given a compact subset $\K\subset\C$, it suffices to show that, for some 
$\delta>0$,
\[
\abs{\frac{K_a(z,t_n^2)}{K_a(t_n^2,t_n^2)}}\abs{\xi(t_n^2,x)}\le C n^{-1-\delta}
\]
for all $z\in\K$, $x\in[0,b]$, and $n$ sufficiently large. 
By Lemma~\ref{lem:denominator}, there exists $n_0\in\N$ such that
\[
\abs{K_a(t_n^2,t_n^2)}\ge C n^{-2\nu-1}
\]
for $n\ge n_0$; we can assume $n_0$ so large that also $t_n^2\not\in\K$ for 
$n\ge n_0$. Furthermore, in view of \eqref{eq:another-estimate} and 
\eqref{eq:asymptotic-eigenvalues},
\[
\abs{\xi(t_n^2,x)}
	\le C n^{-\nu-\frac12}
\]
for all $n\ge n_0$ and $x\in[0,b]$. Therefore, it suffices to
show that
\[
\abs{K_a(z,t_n^2)}\le C n^{-\nu-\frac32-\delta},
\]
for some $\delta>0$.

Abbreviate $\inner{\cdot}{\cdot}:=\inner{\cdot}{\cdot}_{L^2(0,a)}$. We have
\begin{equation}
\label{eq:aliasing-numerator}
\abs{K_a(z,t_n^2)}
	\le \abs{\inner{\xi(\cc{z})}{\xi_\nu(t_n^2)}}
	 +  \abs{\inner{\xi(\cc{z})}{\xi_{\nu,1}(t_n^2)}}
	 +  \abs{\inner{\xi(\cc{z})}{\Xi(t_n^2)}}.
\end{equation}
From \eqref{eq:for-aliasing-1},
\begin{align*}
\inner{\xi(z)}{\xi_\nu(t_n^2)}
	&= \left. \xi(z,a)\xi_{\nu+1}(t_n^2,a)\right.
			\nonumber
	\\[1mm]
	&\qquad 
	+ \left.\frac{1}{t^2}\left(\xi'(z,a)
			- (\nu + \tfrac12)a^{-1}\xi(z,a)\right)\xi_{\nu}(t_n^2,a)\right.
			\nonumber
	\\[1mm]
	&\qquad - \left.\frac{1}{t_n^2}
		\int_0^a(q(x)-z)\xi(z,x)\xi_\nu(t_n^2,x)dx\right..
		\nonumber
\end{align*}
Note that \eqref{eq:asymptotic-eigenvalues} and \eqref{eq:asym-residual-term} imply
\[
t_na = \left(n+\frac{2\nu+1}{4}\right)\pi + \cO(n^{-\frac1p}),
\]
since $\gamma\ne 0$ and $p\in(1,2)$, where $1/r+1/p=1$. This in turn implies
\begin{equation}
\label{eq:another-bound}
\abs{\xi_{\nu}(t_n^2,a)}\le C n^{-\nu-\frac12},\qquad
\abs{\xi_{\nu+1}(t_n^2,a)}\le C n^{-\nu - \frac12 -\frac{1}{p}}.
\end{equation}
Also,
\begin{equation*}
\int_0^a\abs{(q(x)-z)\xi(z,x)\xi_{\nu}(t^2,x)} dx
	\le C n^{-\nu}
\end{equation*}
uniformly for $z\in\K$.
Therefore, there exists $C_1=C_1(\K,\nu,\gamma,a)>0$ such that
\begin{equation*}
%\label{eq:aliasing-first}
\abs{\inner{\xi_\nu(\cc{z})}{\xi_\nu(t_n^2)}}
\le C_1 n^{-\nu - \frac32 - \frac1p},\quad z\in\K,\quad n\ge n_0.
\end{equation*}

We now look at the second term in \eqref{eq:aliasing-numerator}.
As computed in \ref{app:second-computation} (see \eqref{eq:next-to-last}),
\begin{align*}
\inner{\xi(\cc{z})}{\xi_{\nu,1}(t^2)}
	\nonumber
	&= \left.\xi(z,a)Q_1(t^2,a)\xi_{\nu+1}(t^2,a)\right.
		- \frac{1}{t^2}\left.\xi(z,a)Q_2(t^2,a)\theta_{\nu+1}(t^2,a)\right.
	\nonumber
	\\[1mm]
	&\qquad - \frac{1}{t^2}
		\left.\left((\nu + \tfrac12)a^{-1}\xi(z,a) -
		\xi'(z,a)\right)\xi_{\nu,1}(t^2,a)\right.
	\nonumber
	\\[1mm]
	&\qquad - \frac{1}{t^2}
			\int_0^a(q(x)-z)\xi(z,x)\xi_{\nu,1}(t^2,x)dx
	\nonumber
	\\
	&\qquad - \frac{1}{t^2}
			\int_0^a \xi(z,x)q(x)\xi_{\nu}(t^2,x)dx.
	\nonumber
\end{align*}
As a consequence of Lemma~\ref{lem:relief},
\begin{equation*}
\abs{\xi(z,a)}
\abs{Q_1(t^2,a)\xi_{\nu+1}(t^2,a)-t^{-2}Q_2(t^2,a)\theta_{\nu+1}(t^2,a)}
	\le C_2 t^{-\nu -\frac32 - \frac1p},\quad z\in\K,
\end{equation*}
where $1/r + 1/p = 1$. Also,
\begin{equation*}
\abs{(\nu + \tfrac12)a^{-1}\xi(z,a) - \xi'(z,a)}\abs{\xi_{\nu,1}(t^2,a)}
	\le C_2 t^{-\nu-\frac12},\quad z\in\K.
\end{equation*}
By the same argument that leads to \eqref{eq:for-aliasing} 
and \eqref{eq:for-aliasing-too},
\begin{equation*}
\int_0^a\abs{(q(x)-z)\xi(z,x)\xi_{\nu,1}(t^2,x)}dx \le C_2 t^{-\nu},\quad z\in\K
\end{equation*}
and
\begin{equation*}
\int_0^a \abs{\xi(z,x)q(x)\xi_{\nu}(t^2,x)}dx \le C_2 t^{-\nu},\quad z\in\K.
\end{equation*}
Therefore, for all $z\in\C$ and $n\ge n_0$,
\begin{equation*}
\abs{\inner{\xi(\cc{z})}{\xi_{\nu,1}(t_n^2)}}
\le C_2n^{-\nu-\frac32-\frac1p}
\end{equation*}
for some $C_2=C_2(\K,\nu,q,\gamma,a)>0$.

Finally, the last term in \eqref{eq:aliasing-numerator} can be bounded as
\begin{equation*}
\abs{\inner{\xi(\cc{z})}{\Xi(t_n^2)}}
	\le \norm{\xi(\cc{z})}_{L^2(0,a)}\norm{\Xi(t_n^2)}_{L^2(0,a)}
	\le C_3 n^{-\nu- \frac32 -\delta},
\end{equation*}
for some $C_3=C_3(\K,\nu,q,\gamma,a)>0$ and all $n\ge n_0$.
The proof is now complete.
\end{proof}

The proof of the following assertion is nearly identical to the proof of
Lemma~4.2 in \cite{siltoluri}, hence omitted.

\begin{lemma}
\label{propiedades de xi extendida}
Assume that \ref{hyp:convergence-undersampling} is met. Define
\begin{equation*}
\xi^\text{\rm ext}_a(z,x) 
	\defeq \sum_{\lambda_n \in \sigma(H_{a,\gamma})} 
	\frac{K_a(z,\lambda_n)}{K_a(\lambda_n,\lambda_n)}\xi(\lambda_n,x),
	\quad  x \in [0,b], \quad z \in\C.
\end{equation*}
Then, for each $z\in\C$,
\begin{enumerate}[label={(\roman*)}]
\item $\xi^\text{\rm ext}_a(z,\cdot)$ is continuous in $[0,b]$,
	\label{lem:xi-ext-is-continuous}
\item $\xi^\text{\rm ext}_a(z,x) = \xi(z,x)$ for a. e. $x \in [0,a]$, and 
	\label{lem:xi-ext-is-xi}
\item the function
	$h_{ab}(z):=\displaystyle\sup_{x \in [a,b]} \abss{\xi^\text{\rm ext}_a(z,x) 
	- 
	\xi(z,x)}$
	is continuous in $\C$. \label{lem:h-is-continuous}
\end{enumerate}
Moreover,
\begin{enumerate}[resume*]
\item if $F(z)=\inner{\xi(\cc{z})}{\psi}_{L^2(0,b)}$ with $\psi \in L^2(0,b)$, 
then
	\begin{equation}   \label{xi ext prod int}
	\inner{ \xi^\text{\rm ext}_a(\cc{z}) }{\psi}_{L^2(0,b)}  =
	\sum_{\lambda_n \in \sigma(H_{a,\gamma})} 
	\frac{K_a(z,\lambda_n)}{K_a(\lambda_n,\lambda_n)} F(\lambda_n) ,
	\qquad z \in \C .
	\end{equation}
	\label{lem:xi-ext-inner-psi}
\end{enumerate}
\end{lemma}

Our main assertion concerning aliasing follows from Remark~\ref{rem:unitary-invariants},
Proposition~\ref{prop:condition-aliasing-general} and 
Lemma~\ref{propiedades de xi extendida}:

\begin{theorem}
Suppose the hypotheses of Proposition~\ref{prop:condition-aliasing-general}. 
For every $F\in\cB_b$, define
\begin{equation*}
\tilde{F}(z)
%= \inner{\xi^\text{\rm ext}(\cc{z})}{\psi}_{L^2(0,b)}
= \sum_{\lambda_n\in\sigma(S_{a,\gamma})}
	\frac{K_a(z,\lambda_n)}{K_a(\lambda_n,\lambda_n)}\, F(\lambda_n).
\end{equation*}
Then,
\begin{equation*}
\abs{F(z)-\tilde{F}(z)}
	\le h_{ab}(z)\int_a^b\abs{\psi(x)}dx,
\end{equation*}
where $\psi\in L^2(0,b)$ satisfies $F(z)=\inner{\xi(\cc{z})}{\psi}_{L^2(0,b)}$.
\end{theorem}

\begin{remark}
We clearly have
\begin{equation*}
\abs{F(z)-\tilde{F}(z)}
	\le \sqrt{b-a}\,h_{ab}(z)\left(\int_a^b\abs{\psi(x)}^2\right)^{\frac12}dx
	\le C(\K,a,b)\norm{(I-P_{ab})F}_{\cB_b},
\end{equation*}
for all $z\in\K$, where $\K$ is any compact subset of $\C$ and
$P_{ab}$ denotes the orthogonal projector onto $\cB_a$.
%As expected, $C(\K,a,b)\to 0$ as $b\to a$.
\end{remark}

\appendix

\section{Auxiliary results}

\subsection{Computation \#1}
\label{app:first-computation}

Let us recall the identities
\begin{equation*}
\label{eq:identities-bessel-functions}
\cC_\nu(tx) = t^{-1}x^{-\nu - 1}\frac{d}{dx}\left(x^{\nu+1}
				\cC_{\nu+1}(tx)\right),\quad
\cC_{\nu+1}(tx) 
			= - t^{-1}x^{\nu}\frac{d}{dx}\left(x^{-\nu}\cC_{\nu}(tx)\right),
\end{equation*}
where $\cC_\nu$ denotes either $J_\nu$ or $Y_\nu$ (see \cite[Eq.~10.22.1]{nist}).
Applied to \eqref{eq:fundamental-solution-unperturbed-good} and
\eqref{eq:fundamental-solution-unperturbed-bad}, they imply
\begin{equation}
\label{eq:identities-good}
\xi_\nu(t^2,x) = x^{-\nu-\frac12}\partial_x\!%\frac{\partial}{\partial x}
				\left(x^{\nu+\frac12}\xi_{\nu+1}(t^2,x)\right),\quad
\xi_{\nu+1}(t^2,x) = - t^{-2}x^{\nu+\frac12}\partial_x\!%\frac{\partial}{\partial x}
				\left(x^{-\nu-\frac12}\xi_{\nu}(t^2,x)\right),
\end{equation}
and
\begin{equation}
\label{eq:identities-bad}
\theta_\nu(t^2,x) = t^{-2} x^{-\nu-\frac12}\partial_x\!%\frac{\partial}{\partial x}
				\left(x^{\nu+\frac12}\theta_{\nu+1}(t^2,x)\right),\quad
\theta_{\nu+1}(t^2,x) = - x^{\nu+\frac12}\partial_x\!%\frac{\partial}{\partial x}
				\left(x^{-\nu-\frac12}\theta_{\nu}(t^2,x)\right).
\end{equation}
Fix $\epsilon>0$ small. A double integration by parts involving 
\eqref{eq:identities-good} yields
\begin{align*}
\int_\epsilon^a \xi(z,x)\xi_\nu(t^2,x)dx
	&= \left. \xi(z,x)\xi_{\nu+1}(t^2,x)\right|_\epsilon^a
%			\nonumber
%	\\[1mm]
%	&\qquad 
	+ \left.\frac{1}{t^2}\left(\xi'(z,x)
			- (\nu + \tfrac12)x^{-1}\xi(z,x)\right)\xi_{\nu}(t^2,x)\right|_\epsilon^a
			\nonumber
	\\[1mm]
	&\qquad -\left.\frac{1}{t^2}
			\int_\epsilon^a x^{-\nu-\frac12}\partial_x\!%\frac{\partial}{\partial x}
			\left(x^{2\nu+1}\partial_x\!%\frac{\partial}{\partial x}
			\left(x^{-\nu-\frac12}\xi(z,x)\right)\right)
			\xi_{\nu}(t^2,x)dx\right..\nonumber
\end{align*}
Recalling \eqref{eq:boundary-left-condition-limit-circle},
\eqref{eq:A1-kost-sak-tes-unperturb} and \eqref{eq:estimate-xi-one}, taking the
limit $\epsilon\to 0$, and using that $(\tau-z)\xi(z,\cdot)=0$, one obtains
\begin{align}
\inner{\xi(z)}{\xi_\nu(t^2)}_{L^2(0,a)}
	&= \left. \xi(z,a)\xi_{\nu+1}(t^2,a)\right.
			\nonumber
	\\[1mm]
	&\qquad 
	+ \left.\frac{1}{t^2}\left(\xi'(z,a)
			- (\nu + \tfrac12)a^{-1}\xi(z,a)\right)\xi_{\nu}(t^2,a)\right.
			\nonumber
	\\[1mm]
	&\qquad - \left.\frac{1}{t^2}
		\int_0^a(q(x)-z)\xi(z,x)\xi_\nu(t^2,x)dx\right..
	\label{eq:for-aliasing-1}
\end{align}
Similarly,
\begin{align}
\int_a^b \frac{b-x}{b-a}\xi(z,x)&\xi_\nu(t^2,x)dx\nonumber
	\\[1mm]
	&= -\left. \xi(z,a)\xi_{\nu+1}(t^2,a)\right.
%		\nonumber
%	\\[1mm]
%	&\qquad 
	-\left.\frac{1}{t^2}\left(\xi'(z,a)
				-(\nu+\tfrac12)a^{-1}\xi(z,a)\right)\xi_{\nu}(t^2,a)\right.
		\nonumber
	\\[1mm]
	&\qquad +\left.\frac{\xi(z,a)}{t^2(b-a)}\xi_{\nu}(t^2,a)
				-\frac{\xi(z,b)}{t^2(b-a)}\xi_{\nu}(t^2,b)\right.
		\nonumber
	\\[1mm]
	&\qquad -\left.\frac{1}{t^2}
			\int_a^b \frac{b-x}{b-a}(q(x)-z)\xi(z,x)\xi_{\nu}(t^2,x)dx\right.
		\nonumber
	\\[1mm]
	&\qquad - \left.\frac{2}{t^2}
		\int_a^b\frac{1}{b-a}\xi'(z,x)\xi_\nu(t^2,x)dx\right..
	\nonumber%\label{eq:this-too}
\end{align}
Thus,
\begin{align}
\inner{\xi({z})}{R_{ab}\xi_\nu(t^2)}_{L^2(0,b)}
	&= -\,\frac{1}{t^2}\left.\frac{1}{b-a}
		\left(\xi(z,b)\xi_\nu(t^2,b) - \xi(z,a)\xi_\nu(t^2,a)\right)
		\right.
	\nonumber
%	\\
%	&\qquad +\left. 2\int_a^b \xi'(z,x)\xi_\nu(t^2,x)dx\right.
%	\nonumber
	\\
	&\qquad -\frac{1}{t^2}\left.\int_0^bR(x)(q(x)-z)\xi(z,x)\xi_\nu(t^2,x)dx\right.
	\nonumber
	\\[1mm]
	&\qquad - \left.\frac{2}{t^2}
		\int_a^b\frac{1}{b-a}\xi'(z,x)\xi_\nu(t^2,x)dx\right..
	\label{eq:first-result}
\end{align}

\subsection{Computation \#2}
\label{app:second-computation}

Fix some arbitrarily small $\epsilon>0$. 
Using \eqref{eq:identities-good} and integration by parts, one 
obtains
\begin{align}
\int_\epsilon^a \xi(z,x)&Q_1(t^2,x)\xi_\nu(t^2,x) dx
	\nonumber
	\\[1mm]
	&=\left.\xi(z,x)Q_1(t^2,x)\xi_{\nu+1}(t^2,x)\right|_\epsilon^a
	\nonumber
	\\[1mm]
	&\qquad - \frac{1}{t^2}
			\left.\left((\nu+\tfrac12)x^{-1}\xi(z,x) - \xi'(z,x)\right)
				Q_1(t^2,x)\xi_{\nu}(t^2,x)\right|_\epsilon^a
	\nonumber
	\\[1mm]
	&\qquad + \frac{1}{t^2}
			\left.\xi(z,x)q(x)\theta_{\nu}(t^2,x)(\xi_{\nu}(t^2,x))^2\right|_\epsilon^a
	\nonumber
	\\[1mm]
	&\qquad - \frac{1}{t^2}
			\int_\epsilon^a(q(x)-z)\xi(z,x)Q_1(t^2,x)\xi_{\nu}(t^2,x)dx
	\nonumber
	\\[1mm]
	&\qquad - \frac{2}{t^2}
			\int_\epsilon^a\xi'(z,x)q(x)\theta_{\nu}(t^2,x)(\xi_{\nu}(t^2,x))^2dx
	\nonumber
	\\[1mm]
	&\qquad - \frac{1}{t^2}
			\int_\epsilon^a\xi(z,x)Q_1''(t^2,x)\xi_{\nu}(t^2,x) dx.
	\nonumber%\label{eq:second-computation}
\end{align}
Note that we require $q$ to be locally absolutely continuous in order to make sense
of $Q_i''(z,x)$ ($i=1,2$).

Similarly but now using \eqref{eq:identities-bad},
\begin{align}
\int_\epsilon^a \xi(z,x)&Q_2(t^2,x)\theta_\nu(t^2,x) dx
	\nonumber
	\\[1mm]
	&=\frac{1}{t^2}\left.\xi(z,x)Q_2(t^2,x)\theta_{\nu+1}(t^2,x)\right|_\epsilon^a
	\nonumber
	\\[1mm]
	&\qquad - \frac{1}{t^2}
			\left.\left((\nu+\tfrac12)x^{-1}\xi(z,x)-\xi'(z,x)\right)
				Q_2(t^2,x)\theta_{\nu}(t^2,x)\right|_\epsilon^a
	\nonumber
	\\[1mm]
	&\qquad + \frac{1}{t^2}
			\left.\xi(z,x)q(x)\theta_{\nu}(t^2,x)(\xi_{\nu}(t^2,x))^2\right|_\epsilon^a
	\nonumber
	\\[1mm]
	&\qquad - \frac{1}{t^2}
			\int_\epsilon^a(q(x)-z)\xi(z,x)Q_2(t^2,x)\theta_{\nu}(t^2,x)dx
	\nonumber
	\\[1mm]
	&\qquad - \frac{2}{t^2}
			\int_\epsilon^a\xi'(z,x)q(x)\theta_{\nu}(t^2,x)(\xi_{\nu}(t^2,x))^2dx
	\nonumber
	\\[1mm]
	&\qquad - \frac{1}{t^2}
			\int_\epsilon^a\xi(z,x)Q_2''(t^2,x)\theta_{\nu}(t^2,x) dx.
	\nonumber%\label{eq:second-computation-too}
\end{align}
Since $W(\xi_\nu(z),\theta_\nu(z))=1$,
\begin{equation*}
%\label{eq:nice-help}
Q_1''(t^2,x)\xi_{\nu}(t^2,x) - Q_2''(t^2,x)\theta_{\nu}(t^2,x)
	= q(x) \xi_{\nu}(t^2,x).
\end{equation*}
Hence, recalling \eqref{eq:the-other-xi},
\begin{align*}
\int_\epsilon^a\xi(z,x)\xi_{\nu,1}(t^2,x)dx
	\nonumber
	&= \left.\xi(z,x)Q_1(t^2,x)\xi_{\nu+1}(t^2,x)\right|_\epsilon^a
		- \frac{1}{t^2}\left.\xi(z,x)Q_2(t^2,x)\theta_{\nu+1}(t^2,x)\right|_\epsilon^a
	\nonumber
	\\[1mm]
	&\qquad - \frac{1}{t^2}
		\left.\left((\nu + \tfrac12)x^{-1}\xi(z,x) -
		\xi'(z,x)\right)\xi_{\nu,1}(t^2,x)\right|_\epsilon^a
	\nonumber
	\\[1mm]
	&\qquad - \frac{1}{t^2}
			\int_\epsilon^a(q(x)-z)\xi(z,x)\xi_{\nu,1}(t^2,x)dx
	\nonumber
	\\
	&\qquad - \frac{1}{t^2}
			\int_\epsilon^a \xi(z,x)q(x)\xi_{\nu}(t^2,x)dx.
\end{align*}
Clearly,
\begin{equation*}
\lim_{x\to 0+} \xi(z,x)Q_1(t^2,x)\xi_{\nu+1}(t^2,x) = 0,
\end{equation*}
for every $z\in\C$ and $t>0$. Also, \eqref{eq:another-estimate} in conjunction 
with
\begin{equation*}
\theta_\nu(t^2,x) 
	= \frac{\Gamma(\nu+1)2^{\nu-\frac12}}{x^{\nu-\frac12}\nu\pi^{\frac12}}
	\begin{cases}
		g_\nu((tx)^2),& \nu\in\R_+\setminus\N,
		\\[1mm]
		g_\nu((tx)^2) - \frac{\nu(tx)^{2\nu}\log(x)}{\Gamma(\nu+1)^22^{2\nu-1}}
		f_\nu((tx)^2),& \nu\in\N,
	\end{cases}
\end{equation*}
where $f_\nu(0) = g_\nu(0) = 0$ \cite[Eq.~2.7]{kostenko}, imply
\begin{equation*}
\lim_{x\to 0+} \xi(z,x)Q_2(t^2,x)\theta_{\nu+1}(t^2,x) = 0.
\end{equation*}
Moreover, \eqref{eq:boundary-left-condition-limit-circle} and 
\eqref{eq:estimate-xi-one} yield
\begin{equation*}
\lim_{x \to 0+}
	\left((\nu + \tfrac12)x^{-1}\xi(z,x) -\xi'(z,x)\right)\xi_{\nu,1}(t^2,x)= 0.
\end{equation*}
Therefore,
\begin{align}
\inner{\xi(\cc{z})}{\xi_{\nu,1}(t^2)}_{L^2(0,a)}
	\nonumber
	&= \left.\xi(z,a)Q_1(t^2,a)\xi_{\nu+1}(t^2,a)\right.
		- \frac{1}{t^2}\left.\xi(z,a)Q_2(t^2,a)\theta_{\nu+1}(t^2,a)\right.
	\nonumber
	\\[1mm]
	&\qquad - \frac{1}{t^2}
		\left.\left((\nu + \tfrac12)a^{-1}\xi(z,a) -
		\xi'(z,a)\right)\xi_{\nu,1}(t^2,a)\right.
	\nonumber
	\\[1mm]
	&\qquad - \frac{1}{t^2}
			\int_0^a(q(x)-z)\xi(z,x)\xi_{\nu,1}(t^2,x)dx
	\nonumber
	\\
	&\qquad - \frac{1}{t^2}
			\int_0^a \xi(z,x)q(x)\xi_{\nu}(t^2,x)dx.
	\label{eq:next-to-last}
\end{align}

\subsection{Computation \#3}
\label{app:third-computation}

Much in the same vein as before, one obtains
\begin{align*}
\int_a^b \frac{b-x}{b-a}\xi(z,x)&Q_1(t^2,x)\xi_\nu(t^2,x) dx
	\\[1mm]
	&=-\left.\xi(z,a)Q_1(t^2,a)\xi_{\nu+1}(t^2,a)\right.
	\\[1mm]%
	&\qquad + \frac{1}{t^2}
			\left.\left((\nu+\tfrac12)a^{-1}\xi(z,a) - \xi'(z,a)\right)
				Q_1(t^2,a)\xi_{\nu}(t^2,a)\right.
	\\[1mm]
	&\qquad - \frac{1}{t^2}\left.\frac{1}{b-a}
			\xi(z,x)Q_1(t^2,x)\xi_\nu(t^2,x)\right|^b_a
	\\[1mm]
	&\qquad - \frac{1}{t^2}
			\left.\xi(z,a)q(a)\theta_{\nu}(t^2,a)(\xi_{\nu}(t^2,a))^2\right.
	\\[1mm]
	&\qquad - \frac{1}{t^2}
			\int_a^b\frac{b-x}{b-a}(q(x)-z)\xi(z,x)Q_1(t^2,x)\xi_{\nu}(t^2,x)dx
	\\[1mm]
	&\qquad - \frac{2}{t^2}
			\int_a^b\frac{b-x}{b-a}\xi'(z,x)q(x)\theta_{\nu}(t^2,x)(\xi_{\nu}(t^2,x))^2dx
	\\[1mm]
	&\qquad - \frac{1}{t^2}
			\int_a^b\frac{b-x}{b-a}\xi(z,x)Q_1''(t^2,x)\xi_{\nu}(t^2,x) dx
	\nonumber
	\\[1mm]
	&\qquad + \frac{2}{t^2}
			\int_a^b\frac{1}{b-a}\xi(z,x)q(x)\theta_{\nu}(t^2,x)(\xi_{\nu}(t^2,x))^2dx
	\\[1mm]
	&\qquad + \frac{2}{t^2}
			\int_a^b\frac{1}{b-a}\xi'(z,x)Q_1(t^2,x)\xi_{\nu}(t^2,x)dx.
	\nonumber
\end{align*}
Also,
\begin{align*}
\int_a^b \frac{b-x}{b-a}\xi(z,x)&Q_2(t^2,x)\theta_\nu(t^2,x) dx
	\\
	&= -\,\frac{1}{t^2}\left.\xi(z,a)Q_2(t^2,a)\theta_{\nu+1}(t^2,a)\right.
	\\[1mm]
	&\qquad + \frac{1}{t^2}
			\left((\nu + \tfrac12)a^{-1}\xi(z,a)-\xi'(z,a)\right)
				Q_2(t^2,a)\theta_{\nu}(t^2,a)
	\\[1mm]
	&\qquad - \frac{1}{t^2}\left.\frac{1}{b-a}
			\xi(z,x)Q_2(t^2,x)\theta_{\nu}(t^2,x)\right|^b_a
	\\[1mm]
	&\qquad - \frac{1}{t^2}
			\left.\xi(z,a)q(a)\theta_{\nu}(t^2,a)(\xi_{\nu}(t^2,a))^2\right.
	\\[1mm]
	&\qquad - \frac{1}{t^2}
			\int_a^b\frac{b-x}{b-a}(q(x)-z)\xi(z,x)Q_2(t^2,x)\theta_{\nu}(t^2,x)dx
	\\[1mm]
	&\qquad - \frac{2}{t^2}
			\int_a^b\frac{b-x}{b-a}\xi'(z,x)q(x)\theta_{\nu}(t^2,x)(\xi_{\nu}(t^2,x))^2dx
	\\[1mm]
	&\qquad - \frac{1}{t^2}
			\int_a^b\frac{b-x}{b-a}\xi(z,x)Q_2''(t^2,x)\theta_{\nu}(t^2,x) dx
	\\[1mm]
	&\qquad + \frac{2}{t^2}
			\int_a^b\frac{1}{b-a}\xi(z,x)q(x)\theta_{\nu}(t^2,x)(\xi_{\nu}(t^2,x))^2dx
	\\[1mm]
	&\qquad + \frac{2}{t^2}
			\int_a^b\frac{1}{b-a}\xi'(z,x)Q_2(t^2,x)\theta_{\nu}(t^2,x)dx.
\end{align*}
Hence,
\begin{align}
\int_a^b \frac{b-x}{b-a}\xi(z,x)&\xi_{\nu,1}(t^2,x) dx
	\nonumber
	\\[1mm]
	&=-\left.\xi(z,a)Q_1(t^2,a)\xi_{\nu+1}(t^2,a)\right.
		+\frac{1}{t^2}\left.\xi(z,a)Q_2(t^2,a)\theta_{\nu+1}(t^2,a)\right.
	\nonumber
	\\[1mm]
	&\qquad + \frac{1}{t^2}
			\left.\left((\nu+\tfrac12)a^{-1}\xi(z,a) - \xi'(z,a)\right)
				\xi_{\nu,1}(t^2,a)\right.
	\nonumber
	\\[1mm]
	&\qquad - \frac{1}{t^2}\left.\frac{1}{b-a}
			\xi(z,x)\xi_{\nu,1}(t^2,x)\right|^b_a
	\nonumber
	\\[1mm]
	&\qquad - \frac{1}{t^2}
			\int_a^b\frac{b-x}{b-a}(q(x)-z)\xi(z,x)\xi_{\nu,1}(t^2,x)dx
	\nonumber
	\\[1mm]
	&\qquad - \frac{1}{t^2}
			\int_a^b\frac{b-x}{b-a}\xi(z,x)q(x)\xi_{\nu}(t^2,x) dx.
	\nonumber
	\\[1mm]
	&\qquad + \frac{2}{t^2}
			\int_a^b\frac{1}{b-a}\xi'(z,x)\xi_{\nu,1}(t^2,x)dx.
	\label{eq:last-computation}
\end{align}
Therefore, adding \eqref{eq:next-to-last} to \eqref{eq:last-computation},
\begin{align}
\inner{\xi(\cc{z})}{R_{ab}\xi_{\nu,1}(t^2)}_{L^2(0,b)}
%	&=  \lim_{\epsilon\to 0}\int_\epsilon^bR(x)\xi(z,x)\xi_{\nu,1}(t^2,x)dx
%	\nonumber
%	\\[1mm]
	&= -\, \frac{1}{t^2}\left.\frac{1}{b-a}
			\left(\xi(z,b)\xi_{\nu,1}(t^2,b)-\xi(z,a)\xi_{\nu,1}(t^2,a)\right)\right.
	\nonumber
	\\[1mm]
	&\qquad - \frac{1}{t^2}
			\int_0^b R(x)(q(x)-z)\xi(z,x)\xi_{\nu,1}(t^2,x)dx
	\nonumber
	\\[1mm]
	&\qquad - \frac{1}{t^2}
			\int_0^b R(x)q(x)\xi(z,x)\xi_\nu(t^2,x)dx
	\nonumber
	\\[1mm]
	&\qquad + \frac{2}{t^2}
			\int_a^b\frac{1}{b-a}\xi'(z,x)\xi_{\nu,1}(t^2,x)dx.
	\label{eq:last}
\end{align}

\section*{Acknowledgments}
Research partially supported by CONICET (Argentina) through grant PIP 
11220150100327CO.
Part of this work was done while J. H. T. visited IIMAS–UNAM (Mexico). 
He deeply thanks them for their kind hospitality.

%%%%%%%%%%%%%%%%%%%%%%%%%%%%%%%%%%%%%%%%%%%%%%%%%%%%%%%%%%%%%%%%%%%%%%%%%%%%%%%%

\end{document}